\documentclass{amsart}
\usepackage{appendix}
\usepackage{amssymb,amsfonts,amsmath,amsthm}
\usepackage{hyperref}
\usepackage{url}
\usepackage{tabularx}
\usepackage{mathtools}
\usepackage[
backend=biber,
style=alphabetic,
sorting=nyt
]{biblatex}
\newcommand{\omminom}{\omega^{<\omega}}
\newcommand{\baire}{\omega^{\omega}}

\makeatletter
\def\subsection{\@startsection{subsection}{3}%
  \z@{.5\linespacing\@plus.7\linespacing}{.3\linespacing}%
  {\centering}}
\makeatother

\makeatletter
\def\subsubsection{\@startsection{subsubsection}{3}%
  \z@{.5\linespacing\@plus.7\linespacing}{.3\linespacing}%
  {\centering}}
\makeatother

\makeatletter
\def\myfnt{\ifx\protect\@typeset@protect\expandafter\footnote\else\expandafter\@gobble\fi}
\makeatother

\newtheorem{theorem}{Theorem}[section]

\newtheorem{corollary}[theorem]{Corollary}
\newtheorem{definition}[theorem]{Definition}

\newtheorem{proposition}[theorem]{Proposition}

\newtheorem{problem}{Problem}

\newtheorem{question}{Question}
\newtheorem{fact}[theorem]{Fact}
\newtheorem{remark}[theorem]{Remark}
\newtheorem{notation}[theorem]{Notation}

\newtheorem{conjecture}{Conjecture}

\newtheorem*{theorem1.1}{Theorem 1.1}
\newtheorem*{theorem1.3}{Theorem 1.3}
\newtheorem*{theorem1.4}{Theorem 1.4}
\newtheorem*{proposition5.1}{Proposition 5.1}

\begin{document}

\begin{abstract} 
Motivated by the results for \emph{Magic: The Gathering} presented in \cite{mtgTuringComp} and \cite{HardArithmetic}, we study a (different) computability problem about winning strategies in \emph{Yu-Gi-Oh! Trading Card Game}, a popular card game developed and published by \emph{Konami}. We show that the problem of establishing whether, from a given game state, a given computable strategy is winning is undecidable. In particular, not only do we prove that the Halting Problem can be reduced to this problem, but also that this problem is actually $\Pi^1_1$-complete. We extend this last result to all strategies with a reduction on the set of countable well orders, a classic $\boldsymbol{\Pi}^1_1$-complete set. 
For these reductions we present two legal decks (according to the current \emph{Forbidden \& Limited List} of \emph{Yu-Gi-Oh! Trading Card Game}) that can be used by the player who goes first to perform them. 
\end{abstract}

\title[]{Deciding winning strategies in Yu-Gi-Oh! TCG is hard}



\author{Orazio Nicolosi}
\address{Department of Mathematics ``Giuseppe Peano'', University of Torino, Via Carlo Alberto 10, 10123, Torino Italy.}
\email{orazio.nicolosi@unito.it}

\author{Federico Pisciotta}
\address{Department of Mathematics, Computer Science and Physics, University of Udine, Via delle Scienze 206, 33100, Udine Italy.}
\email{pisciotta.federico@spes.uniud.it}

\author{Lorenzo Bresolin}
\address{Department of Mathematics, University of Pisa, Largo Bruno Pontecorvo 5, 56127, Pisa Italy.}
\email{lorenzo.bresolin@phd.unipi.it}

\maketitle


\section{Introduction}

\noindent \emph{Yu-Gi-Oh! Trading Card Game (TCG)} is a popular trading card game owned by Konami and based on the manga series \emph{Yu-Gi-Oh!}. It is a turn-based, competitive strategy game in which participants utilize customizable decks of cards to reduce their opponent’s Life Points to zero or to force them to ``deck out", while maintaining strategic control of the game state. The game is a two-player, non-zero-sum, stochastic game with imperfect information, characterized by a complex interaction of card types, summoning methods, timing windows, and chained effect resolutions.

\noindent Despite its popularity, the authors are not aware of any formal academic literature concerning the computational or logical properties of \emph{Yu-Gi-Oh! TCG}. In contrast, the game \emph{Magic: The Gathering}, another popular card game owned by Wizards of the Coast, has received significant attention. The problem of optimal play in \emph{Magic: The Gathering} has been studied from a probabilistic perspective (see, for example, \cite{Ward2009MonteCS} and \cite{Ward2012MonteCS}) and from the viewpoint of computability theory. In particular, the work of Churchill, Biderman, and Herrick \cite{mtgTuringComp} proves that determining optimal play in \emph{Magic: The Gathering} is an undecidable problem. Moreover, Biderman \cite{HardArithmetic} shows that the problem is at least as hard as arithmetic.

\noindent Motivated by these results, we ask whether an analogous phenomenon occurs in the \emph{Yu-Gi-Oh! TCG}. However, unlike \emph{Magic: The Gathering}, \emph{Yu-Gi-Oh! TCG} does not appear to have the same degree of automatism built into its card interactions. For this reason, we slightly reformulate the problem: instead of considering whether a given position is winning for one of the players, we ask whether a strategy starting from a given configuration guarantees victory.

\begin{question}\label{mainproblem}
Can we determine whether a strategy in the \emph{Yu-Gi-Oh! TCG} is winning?
\end{question}

\noindent Our paper aims to fill this gap in the literature by adapting, in part, the ideas presented in \cite{mtgTuringComp}. In that paper, the authors embed a universal Turing machine into the game in such a way that Player~1 may initiate a computation. In \emph{Yu-Gi-Oh! TCG}, while we do not seem to have this same degree of automatism, we can make the game endless and store arbitrarily large amounts of information. To do so, we introduce two additional assumptions: first, on each turn, a player may perform only finitely many operations; second, there is no time constraint on the duration of the match. We show that, with appropriate adjustments, Player~1 can follow a strategy that simulates the computation of any Turing machine. This is the key observation allowing us to give a negative answer to Question~\ref{mainproblem}.

\noindent From a technical point of view, following the discussion in \cite{mtgTuringComp}, we believe that the \emph{Yu-Gi-Oh! TCG} is \emph{transition-computable}, in the sense that the function mapping a game configuration together with a legal move to the resulting configuration is computable. We point out that this claim is more reasonable for \emph{Yu-Gi-Oh! TCG} than for \emph{Magic: The Gathering}, since card effects in \emph{Yu-Gi-Oh! TCG} are typically simpler and offer fewer choices. Nevertheless, establishing this rigorously appears difficult without an exhaustive analysis of the more than 13,000 cards currently in the game. We therefore state the following as a conjecture:

\begin{conjecture}
The function that maps a game configuration and a legal move to the consequent configuration in the \emph{Yu-Gi-Oh! TCG} is computable.
\end{conjecture}

\noindent In the first instance, we restrict our attention to \emph{computable strategies}. Our initial idea is to reduce to this problem the classical Halting Problem. Moreover, since we wish to work within the space of natural numbers~$\omega$ (rather than uncountable spaces, which require different techniques), we restrict ourselves to computable strategies (i.e., the ones that can be performed by computer programs), of which there are only countably many.

\noindent After a precise formalization of the problem, we obtain the following:

\begin{theorem}\label{maintheorem}
Determining whether a computable strategy is winning in \emph{Yu-Gi-Oh! TCG} is an undecidable problem.
\end{theorem}

\noindent As a corollary, this yields an impossibility result for computer algorithms:

\begin{corollary}
No computer program can decide whether a strategy is winning in \emph{Yu-Gi-Oh! TCG}.
\end{corollary}

\noindent From the viewpoint of logic, it is also natural to ask about the exact complexity of the problem. Without entering into a detailed exposition of the lightface analytic hierarchy, recall that this hierarchy stratifies sets according to the alternation of universal and existential quantifiers in their definitions. We redirect to \cite{Montalban_2026} for a brief explanation of the lightface arithmetical hierarchy (contained in the introduction) and of the first level for the lightface analytical hierarchy (contained in Chapter IV). By giving a natural reduction from a $\Pi^1_1$-complete set to our problem, we obtain the following.

\begin{theorem}\label{realcomplexitytheorem}
Determining whether a computable strategy in \emph{Yu-Gi-Oh! TCG} is winning is a $\Pi^1_1$-complete problem.
\end{theorem}

\noindent Finally, extending the analysis to non-computable strategies is almost straightforward. For this result, it could be useful to the reader to have some knowledge also in classical Descriptive Set Theory, in particular of Chapter II.22 and of Chapter III.27 from \cite{kechris}. We present a Wadge-reduction from the set of well-orders on a countable domain, a well-known $\boldsymbol{\Pi}^1_1$-complete set (see \cite[Theorem 27.12]{kechris}), to our problem. Thus, we obtain:

\begin{theorem}
Determining whether a (not necessarily computable) strategy in \emph{Yu-Gi-Oh! TCG} is winning is a $\boldsymbol{\Pi}^1_1$-complete problem.
\end{theorem}

We believe that our techniques could likely be adapted to show the same result about winning stetegies in \emph{Magic: The Gathering}.
It would be interesting to know whether the same holds for other card games (like \emph{Pokémon TGC}).  

\smallskip

\noindent The structure of the paper is as follows. In Section~\ref{secTuring}, we show that, using a specific deck, Player~1 can play according to the instructions of any Turing machine. In Section~\ref{secUnd}, we provide a precise formalization of the problem and prove that it is undecidable. In Section~\ref{secMoreFlex}, we present a variant of the previous deck that gives Player~2 the ability to choose a natural number whenever Player~1 requests it. In Section~\ref{secCoCompl}, we show that it is $\Pi^1_1$-complete. Finally, in Section~\ref{secDST}, we consider the problem for strategies that can be non-computable, and we classify it, proving (not so surprisingly) that in this case the problem is $\boldsymbol{\Pi}^1_1$-complete (or co-analytic complete). In the appendices, we give proof of a result used in Section~\ref{secCoCompl}.
\section{The main construction}\label{secTuring}





Unlike \cite{mtgTuringComp}, to get our result, we do not embed a particular universal Turing machine into a play of \emph{Yu-Gi-Oh! TCG}. Instead, we code the configurations for any Turing machine in the number of spell counters, and its program will be handled by the strategy. In particular, we show that can be reached a legal configuration from which one player can store an arbitrary number each turn and can play following as strategy the instructions of any Turing machine while considering the stored number as an encoding of the configuration of such Turing machine (i.e.\,the content on the tape, the position of the head, and the current state)\footnote{This information can be coded as finite strings in $\omega^{<\omega}$, moreover, we chose by convention that the encoding assigns $0$ to the empty string.}. This suffices to reduce the Diagonal Halting Problem to the considered problem, which thus is (at most) undecidable. 

\subsection{The board}\label{subsecSet}
We now describe the configuration that we desire, and which the player going first (whom we call Player 1) can reach. Player~2 has no cards in their hand, on the field, in the Graveyard, or banished (except for the trap card \textsl{Pole Position} active on their side of the field). 
At the same time, Player~1 controls on the Field Zone, the card \textsl{Magical Cytadel of Endymion} with no spell counters, the continuous spell \textsl{Temple of the Kings} activated, and the trap card \textsl{Localized Tornado} set face-down. On their field, they also have two copies of \textsl{Magician of Faith}, one copy of \textsl{Mask of Darkness}, and one \textsl{Vanity's Ruler}, all faced up. Among the cards in their graveyard, Player~1 needs the following two: the Effect Monster \textsl{Endymion, the Master Magician}, and the Quick-Play Spell \textsl{Offerings to the Doomed}. 
In their hand, they have the following five cards: \textsl{Yata-Garasu}, \textsl{Bait Doll}, \textsl{Upstart Goblin}, \textsl{Book of Eclipse}, and \textsl{Desert Sunlight}. In the Extra Deck, as in the Main Deck, there are no cards.

\subsection{How to set the board}
Before proving that with this board we can emulate the operations of any Turing machine (and keep track of them), we prove that this is legal (i.e., it can be reached from the start of the game if the first player wants). An example of a legal (i.e., which satisfies the condition indicated in \cite{rulebook}) 43-card deck that is capable of reaching this configuration in a game (potentially against any deck) can be seen in the following Table.

\begin{table}[h!]
\scalebox{0.85}{ \centering
\begin{tabular}{|l|l|l|}
\hline
\textbf{Monsters} & \textbf{Spell} & \textbf{Trap} \\
\hline
\textbf{Monsters in Main Deck} & 3 Double Summon & 1 Soul Drain \\
1 Protector of the Sanctuary & 3 Pot of Benevolence & 2 Desert Sunlight \\
1 Morphing Jar  & 1 Gold Moon Coin & 3 Localized Tornado \\
1 Yata-Garasu & 3 Soul Release  & 1 Massivemorph \\
1 Vanity’s Ruler & 3 Rain of Mercy & 2 Gift Card  \\
1 Endymion, the Master Magician & 3 Upstart Goblin  & 1 Pole Position \\
3 Magician of Faith &  1 Mystical Space Typhoon &  \\
1 Mask of Darkness  & 1 Pot of Desires  & \\
\cline{1-1}
\textbf{Monsters in Extra Deck} &  1 Pot of Duality & \\
 & 1 Offerings to the Doomed & \\
 & 1 Magical Citadel of Endymion & \\
  & 1 Temple of the Kings & \\
  & 1 Bait Doll & \\
& 1 Book of Eclipse & \\
\hline
\end{tabular}
\par} 
\caption{A possible deck}
\end{table}

To set the board, the player needs:
\begin{itemize}
    \item Going first.
    \item The correct draw on the first hand (actually, the correct order of the other cards in the deck is important too) made of \textsl{Double Summon}, \textsl{Protector of the Sanctuary}, \textsl{Morphing Jar}, \textsl{Soul Drain}, and \textsl{Desert Sunlight}.
    \item On the first turn, they have to summon face-down the two monsters (using \textsl{Double Summon}) and set the two trap cards. Then, on the Standby Phase of the opponent, they have to activate \textsl{Desert Sunlight} with \textsl{Soul Drain} chained.
    In this way, \textsl{Desert Sunlight} will force the activation of \textsl{Morphing Jar} and both player have to discard their entire hands. However, thanks to \textsl{Protector of the Sanctuary}, the opponent will not be able to draw any cards (and if they have not chained any effect to the previous two traps to stop that strategy, they cannot play unless summoning monsters from the Extra Deck).
    \item The correct draw on the new five cards is: \textsl{Yata-Garasu}, \textsl{Vanity’s Ruler}, \textsl{Double Summon}, and two copies of \textsl{Pot of Benevolence}.
    With them, it can normal summon \textsl{Yata-Garasu} and \textsl{Vanity’s Ruler} (blocking in this way also the opponent's Extra Deck), and return to the opponent's Deck four of the cards sent to the Graveyard. Then, attacking in the Battle Phase with \textsl{Yata-Garasu}, we prevent the opponent from drawing new cards.
    \item From now on, on each turn Player~1 will summon and attack with \textsl{Yata-Garasu}, and will activate the card that they draw if it is a magic card, and a possible order for the cards is: \textsl{Pot of Benevolence} (to leave the opponent without cards in the GY), \textsl{Massivemorph}, \textsl{Pole Position}, \textsl{Gold Moon Coin} (to give the previous two trap cards to the opponent player\footnote{To lock completely the opponent's Extra Deck, we suppose that they play these trap cards targeting \textsl{Vanity's Ruler} with the first one (as we need to make \textsl{Vanity's Ruler} immune by spell cards, as we will use \textsl{Book of Eclipse} for our construction). Since we are considering legal configurations, this is not a problem for our goal.}), \textsl{Endymion, the Master Magician}, \textsl{Localized Tornado}, \textsl{Mystical Space Typhoon} (to get rid of \textsl{Soul Drain}), \textsl{Localized Tornado}, \textsl{Rain of Mercy}, \textsl{Soul Release} (to reduce the number of cards that Player 1 have in the graveyard and hence that will be reshuffled in the deck after using \textsl{Localized Tornado}\footnote{The cards to remove are: \textsl{Protector of the Sanctuary}, \textsl{Morphing Jar}, \textsl{Soul Drain}, \textsl{Desert Sunlight}, and \textsl{Mystical Space Typhoon}.}), \textsl{Pot of Desires} (to reduce the deck of 10 other cards and draw two additional cards -the right cards to draw are \textsl{Offerings to the Doomed} and \textsl{Double Summon} which the player will not activate-).
    \item At this point of the match, in the deck remain only 12 cards that we need to set the board. The following table indicates which card the player draw and which action they have to do on the same turn (in addition to summon and attack with \textsl{Yata-Garasu})\footnote{We recall that in \emph{Yu-Gi-Oh! TCG} if a player has 7 or more cards has to discard until they have only 6 cards, so we take into consideration this constraint.}:

    \begin{table}[h!]
        \scalebox{0.8}{
        \centering
        \begin{tabular}{|m{4.5cm}|m{10cm}|}
            \hline
            \textbf{Drawn card} & \textbf{Action to do} \\
            \hline
            \textsl{Magical Citadel of Endymion} & Set it face-down.\\
            \hline
            \textsl{Temple of the Kings} & Set it face-down.\\
            \hline
            \textsl{Magician of Faith} & Use \textsl{Double Summon} and summon it face-down.\\
            \hline
            \textsl{Magician of Faith} & Flip the other copy of \textsl{Magician of Faith} to recycle \textsl{Double Summon} from the GY so that the new copy can be summoned face-down.\\
            \hline
            \textsl{Mask of Darkness} & Again, flip \textsl{Magician of Faith} to recycle \textsl{Double Summon} from the GY so that \textsl{Mask of Darkness} can be summoned.\\
            \hline
            \textsl{Bait Doll} & \\
            \hline
            \textsl{Upstart Goblin} & At the end of the turn, discard one copy of \textsl{Localized Tornado}.\\
            \hline
            \textsl{Book of Eclipse} & Set it face-down.\\
            \hline
            \textsl{Desert Sunlight} & Again, at the end of the turn, discard one copy of \textsl{Localized Tornado}.\\
            \hline
            \textsl{Localized Tornado} & Set it face-down.\\
            \hline
        \end{tabular}
        }
    \end{table}
    \item At this point, Player~1 has to activate first \textsl{Temple of the Kings} and then \textsl{Magical Citadel of Endymion}, and so the desired configuration is reached.
\end{itemize} 

\subsection{How to increase/decrease the counter}

We now explain how Player~1 can increase/decrease by any quantity the Spell Counters on the Field Spell \textsl{Magical Citadel of Endymion}.
To increase them, we have to activate spell cards, and in particular, we will use the board to recycle three spell cards and a trap card: \textsl{Bait Doll}\footnote{Which actually put itself in the deck after the activation.}, \textsl{Upstart Goblin}, \textsl{Book of Eclipse}, and \textsl{Desert Sunlight}.
Activating any of the spell cards listed before increases the number of Spell Counters only by 1, but the following order of actions allows the player to repeat it an unlimited number of times. First, activate \textsl{Book of Eclipse} to put all monsters on the field face-down, then set \textsl{Desert Sunlight} and activate \textsl{Bait Doll} on it. This would force the activation of the trap card, whose timing activation is right thanks to the presence of \textsl{Temple of the Kings}, and so will face-up in defense position all monsters on the field, activating their flip effects. In this way, the player can recycle from the GY: \textsl{Book of Eclipse} and \textsl{Upstart Goblin}\footnote{We do so only if it is in the GY; otherwise does not matter which spell card is recycled.} (thanks to the two copies of \textsl{Magician of Faith}), and \textsl{Desert Sunlight} (thanks to \textsl{Mask of Darkness}). Finally, using \textsl{Upstart Goblin}, the player can draw the card \textsl{Bait Doll} used during the loop (that returned to the deck after its activation) and restart the loop. One complete cycle increases the number of Spell Counters by 3, and since, as we will see below, we may assume that the player skips their next Draw Phase, this cycle can be stopped at any point. So the number of Spell Counters can be increased as needed. To decrease the number of Spell Counters we use (or better, we summon) \textsl{Endymion, the Master Magician}, which can be special summoned by removing 6 Spell Counters from the \textsl{Magical Citadel of Endymion}. Actually, we also need to remove this monster from the field to resummon it again. To this extent, we use the card \textsl{Offerings to the Doomed}, which allows us to destroy \textsl{Endymion, the Master Magician}, and to skip the next Draw Phase, although it increases the Spell Counters by 1. As \textsl{Endymion, the Master Magician} can also be summoned from the GY, and each time that is summoned, we can recycle any spell from the GY, we can repeat this loop, decreasing each time by 5 Spell Counters until remain only less than 6 such counters.
It is clear that combining the two loops (which do not interfere), we can modify the number of counters to any number in $\mathbb{N} \setminus \{0\}$, leaving on the field only the three flip monsters and \textsl{Vanity’s Ruler} (so leaving space to summon \textsl{Yata-Garasu})\footnote{As the encoding that we have chosen for the finite strings assigns $0$ to the empty strings and we will never deal with the empty string (since any Turing machine always has a state), this is not a restriction.}. Moreover, the player can execute at the start of each turn both loops to increase the Spell Counters by 1 and skip the next Draw Phase.
\section{The game is undecidable}\label{secUnd}

In order to answer Question~\ref{mainproblem}, we first need to formalize it as mathematical problem. 
Two ingredients are required.  

\medskip

The first one is the \emph{configuration}. Informally, it can be seen as a ``screenshot'' of the game from the point of view of one of the two players (for example,   Player~1).

\begin{definition}
    A \emph{legal configuration} $C$ is an element of $\omminom$ that encodes all the information known by one player about the current state of the game.
\end{definition}
\begin{notation}
    We write $\mathcal{C}$ for the set of all legal configurations.
\end{notation}
The definition above is somewhat informal because the amount of information involved is quite large and would be too pedantic to write down explicitly. 
For example, a detailed but not exhaustive list of all these elements includes: the Life Points of both players, the current Phase, the number of turns played, the number of cards in each Deck, in the Graveyards, and in the Extra Decks, as well as the cards and their positions in the Monster Zones and Spell \& Trap Zones, and the effects of all the cards that are active on the current Phase (even if they were activated in a previous turn).
We also encode the number of counters on each card, where tokens are placed on the Field, and whether the player knows any cards in the opponent’s Deck, Hand, or Field.

\begin{remark}
    Each card can be represented by its unique identifier in a fixed enumeration of all \emph{Yu-Gi-Oh! TCG} cards.
\end{remark}

Since the amount of information is always finite, a configuration can indeed be represented as (or coded by) a finite string, i.e., an element of $\omminom$. Between the legal configurations, we call \emph{initial configurations} the legal configurations that correspond to possible configurations that start a match in \emph{Yu-Gi-Oh! TCG}, i.e., with both players with 8000 Life Points, with no banned cards in the GY, with two decks with the right number of cards, etc...

\medskip

The second ingredient is a precise formulation of the notion of \emph{strategy}. 

A strategy is closely related to the concept of a \emph{legal run}, but before this, we need to introduce the following definition:

\begin{definition}
    A \emph{legal move} is any action that a player is allowed to perform according to the official Rulebook (see \cite{rulebook}) from a legal configuration (and which changes it).
\end{definition}

\begin{definition}
    A \emph{legal run} is a sequence of legal configurations $\overline{C} \in \mathcal{C}^{<\omega}$ such that:
    \begin{itemize}
        \item $C_0$ is an initial configuration
        \item $C_{i+1}$ can be obtained from $C_i$ by a legal move (performed by either of the two players) for any $i < \operatorname{l}(\overline{C})-1$. 
    \end{itemize}
\end{definition}

\begin{notation}
    We write $\mathcal{R}$ for the set of all legal runs.
\end{notation}

\begin{remark}
    The set $\mathcal{R}$ is $\Sigma^0_1$, indeed:
    \begin{align*}
    \overline{C} \in \mathcal{R} \iff & C_0 \text{ is an initial configuration } \land \forall j < \operatorname{l}(\overline{C}) (C_j \in \mathcal{C}) \land \\
    & \forall i < \operatorname{l}(\overline{C})-1 \exists M \in \omega ( M \text{ codes a move } \land t(C_i, C_{i+1}, M) )    
    \end{align*}
    where $t$ is the computable predicate that checks wheter a transition between two legal configurations is legit (which we assume to exists). Moreover, notice that: 
    \begin{itemize}
        \item As there is only a finite number of initial configurations, the first condition is computable.
        \item Similarly, as there are only a finite number of cards and the field is finite, checking if a configuration is in $\mathcal{C}$ is computable.
        \item Finally, as there are only finitely many moves from any configuration, we can code them using natural numbers in a computable way.
    \end{itemize}
\end{remark}

We now describe our notion of a strategy. This is based on two main ideas.  
First, a player’s strategy must take into account everything that has happened during the game; thus, it must retain memory of all configurations encountered so far.  
Second, at the moment, we consider only strategies implemented by a computer. Therefore, we require that a strategy be a computable function. 

\begin{definition}
    A computable \emph{strategy for Player 1} is a computable partial function
    $$
        f : \mathcal{R} \rightharpoonup \mathcal{C}, \quad 
        (C_1, \dots, C_n) \mapsto C
    $$
    such that $C$ is a legal configuration that can be reached from $C_n$ by Player 1 performing one legal move.
\end{definition}
One can similarly define the strategy for Player 2 (i.e., the one starting as the second player).
\begin{notation}
    We denote with $\mathcal{S}$ the set of all computable strategies.
\end{notation}

\begin{remark}
    During the opponent’s turn, the strategy may allow the player to pass or to activate a Trap Card, a Quick-Play Spell Card, or an ``Hand Trap".
\end{remark}
Now we are ready to define the problem in a proper way.

\begin{definition}
Let $\overline{C}\in\mathcal{R}$ and $S \in \mathcal{S}$. We say that a computable strategy $S \in \mathcal{S}$  starting from $\overline{C}$ is \emph{winning for Player~1} if it leads to the victory of the game, regardless of Player 2's moves.
\end{definition}

Again, a similar definition can be given for Player 2 (the player going second). But, as we only consider (winning) strategies for the starting player, from now on, we do not specify it.
\begin{notation}
A strategy $S \in \mathcal{S}$ starting from $\overline{C} \in \mathcal{R}$ that is winning is denoted by $\overline{C} \overset{S}{\rightarrow}V$.
\end{notation}

\begin{problem}
    Given a legal run $\overline{C} \in \mathcal{R}$ and a computable strategy $S \in \mathcal{S}$, it is decidable whether $\overline{C} \overset{S}{\rightarrow}V$?
\end{problem}

\subsection{The reduction of the halting problem}

\begin{definition}
    Let $e \in \omega$.We call \emph{$e$-run}, and we denote with $\mathbf{C}_e$, the legal run exposed in  Subsection \ref{subsecSet}, which leads to the configuration with the number of Spell Counter on the \textsl{Magical Citadel of Endymion}, which corresponds to the input $e$ on the tape of the $e$-th Turing machine.
\end{definition}

\begin{definition}
    Let $e \in \omega$. We call \emph{$e$-strategy}, and we denote with $\mathbf{S}_e$, the following computable strategy. From the run $\mathbf{C}_e$ the strategy follows the computation of the $e$-th Turing machine starting from the input $e$, using the coding explained in Section \ref{secTuring}, while summoning and attacking each turn with \textsl{Yata-Garasu}. If the computation halts, then Player~1 attacks with all monsters until the opponent is defeated.
\end{definition}

\begin{remark}
    These strategies are well-defined in the sense that they can be executed by Player~1, since the opponent cannot draw any new cards by the effect of the monster card \textsl{Yata-Garasu}.
\end{remark}

\begin{theorem1.1}
Deciding whether a computable strategy in \emph{Yu-Gi-Oh! TCG} is winning is an undecidable problem.
\end{theorem1.1}

\begin{proof}
We prove that the Halting Problem (which is $\Sigma^0_1$-complete) many-reduces to $\{(\overline{C}, S) \mid \overline{C} \overset{S}{\rightarrow}V \}$ and, therefore, this problem is undecidable (as $\Sigma^0_1$-hard).\\
Consider the function $f: \omega \to \mathcal{R} \times \mathcal{S}$ that associate to each $e$ the pair $(\mathbf{C}_e,\mathbf{S}_e)$. Notice that this function is computable, as we suppose that \emph{Yu-Gi-Oh! TCG} is
transition-computable and the moves coded by $\mathbf{S}_e$ are simply the translations of the instructions of the Turing machine with code $e$ together with the instruction to attack if the computation halts. Hence, said $K$, the Diagonal Halting Problem, this function witnesses that $K \le_m \{(\overline{C}, S) \mid \overline{C} \overset{S}{\rightarrow}V \}$, thus the result follows.
\end{proof}
\section{A more flexible construction}\label{secMoreFlex}

We present here a slightly different deck that will be used in the next section to prove that the set $\{(\overline{C}, S) \mid \overline{C} \overset{S}{\rightarrow}V \}$ is $\Pi^1_1$-complete. 

\subsection{The board}

The configuration that we desire to reach is similar to the one explained in section \ref{secTuring} (especially for Player~1); however, it is more complicated, as for this construction, the goal is to allow the opponent (Player~2) to choose numbers that interact with what Player~1 is playing. In particular, the configuration for Player~1 is the same as in section \ref{secTuring}, but in addition: on the Extra Monster Zone (no matter which) it controls \textsl{Starving Venemy Dragon}, they have on the field the continuous spell card \textsl{Attraffic Control} active, and in their hand they have the card \textsl{Raigeki}. At the same time, Player~2 has\footnote{Tecnically we will see that Player~2 is forced to have these cards by Player~1, and in particular, all cards they control are from Player~1's deck.} two copies of \textsl{Flint Lock}, and one copy each of \textsl{Protector of the Sanctuary} and \textsl{Clara \& Rushka, the Ventriloduo} (placed not in the Extra Monster zone). On the Spell \& Trap Zones, they control the equip cards \textsl{Flint} and \textsl{Mist Body} equipped to \textsl{Clara \& Rushka, the Ventriloduo}, the continuous trap card \textsl{Pole Position}, and the continuous spell card \textsl{Morale Boost}. Moreover,they have no cards in hand, in the Graveyard, or banished.

\subsection{How to set the board}

We now exhibit a legal 48-card deck for Player~1 that is capable of reaching this configuration in a game against any deck (if the opponent does not interrupt the play). That is, this deck witnesses that the desired configuration is legal.
 
\begin{table}[h!]
\scalebox{0.85}{ \centering
\begin{tabular}{|l|l|l|}
\hline
\textbf{Monsters} & \textbf{Spell} & \textbf{Trap} \\
\hline
\textbf{Monsters in Main Deck} & 2 Double Summon & 1 Soul Drain \\
1 Protector of the Sanctuary & 3 Pot of Benevolence & 1 Desert Sunlight \\
1 Morphing Jar  & 1 Mystical Space Typhoon & 3 Localized Tornado \\
1 Yata-Garasu & 1 Morale Boost  & 3 Give and Take \\
1 Vanity’s Ruler & 1 Mist Body & 1 Reverse Reuse  \\
1 Endymion, the Master Magician & 3 Gold Moon Coin & 1 Pole Position \\
2 Flint &  1 Magical Stone Excavation & 2 Massivemorph \\
2 Magician of Faith & 1 Attraffic Control  & \\
1 Odd-Eyes Pendulum Dragon & 1 Creature Swap & \\
1 Mask of Darkness & 1 Polymerization & \\
\cline{1-1}

 \textbf{Monsters in Extra Deck} &  1 Flint & \\
 1 Clara \& Rushka, the Ventriloduo & 1 Raigeki & \\
 1 Starving Venemy Dragon & 1 Offerings to the Doomed & \\
& 1 Magical Citadel of Endymion & \\
& 1 Temple of the Kings & \\
& 1 Bait Doll & \\
& 1 Upstart Goblin & \\
& 1 Book of Eclipse & \\
\hline
\end{tabular}
\par} 
\caption{A 48-card deck to many-reduce a $\Pi^1_1$-complete set to our problem}
\end{table}

To set the board, the first part of the instructions (i.e., up to the first 18 cards in the deck) is identical to the ones explained in section \ref{secTuring}. As before, Player~1, who is going first, will summon and attack each turn with \textsl{Yata-Garasu}. The next 19 cards are focused on building the opponent's board. We list them with the corresponding action aside:

\begin{table}
        \scalebox{0.8}{
        \centering
        \begin{tabular}{|m{4.5cm}|m{10cm}|}
            \hline
            \textbf{Drawn card} & \textbf{Action to do} \\
            \hline
            \textsl{Flint Lock} & \\
            \hline
            \textsl{Flint Lock} & \\
            \hline
            \textsl{Magical Stone Excavation} & Activate it by sending the previous two monster cards in the GY and recycling a copy of \textsl{Double Summon}.\\
            \hline
            \textsl{Magician of Faith} & Use \textsl{Double Summon} and summon it face-down.\\
            \hline
            \textsl{Mist Body} & \\
            \hline
            \textsl{Boost Morale} & \\
            \hline
            \textsl{Attraffic Control} & Activate it, flip \textsl{Magician of Faith} to recycle \textsl{Double Summon} from the GY, then special summon \textsl{Clara \& Rushka, the Ventriloduo}.\\
            \hline
            \textsl{Reverse Reuse} & Set it face-down.\\
            \hline
            \textsl{Give and Take} & Set it face-down.\\
            \hline
            \textsl{Give and Take} & Set it face-down.\\
            \hline
            \textsl{Give and Take} & Set it face-down.\\
            \hline
            \textsl{Gold Moon Coin} & \\
            \hline
            \textsl{Creature Swap} & Activate \textsl{Reverse Reuse} to summon face-down on opponent's field \textsl{Magician of Faith}, and then use \textsl{Creature Swap} to change it with \textsl{Clara \& Rushka, the Ventriloduo}. 
            Activate \textsl{Gold Moon Coin} to give \textsl{Morale Boost} and \textsl{Mist Boby} to your opponent. We suppose that they activate them (in particular, the Equip Card needs to be equipped to the link monster).
            Finally, activate in chain all the copies of \textsl{Give and Take} summoning on the opponent's field, \textsl{Protector of the Sanctuary}, and two copies of \textsl{Flint Lock} all face-up in defense position. \\
            \hline
            \textsl{Odd-Eyes Pendulum Dragon} & Again, flip \textsl{Magician of Faith} to recycle \textsl{Double Summon} from the GY. \\
            \hline
            \textsl{Polymerization} & Activate it to summon \textsl{Venemy Starving Dragon} using \textsl{Odd-Eyes Pendulum Dragon} and \textsl{Endymion, the Master Magician} as fusion material. Then use the effect of \textsl{Venemy Starving Dragon} to negate the effects of \textsl{Protector of the Sanctuary} controlled by the opponent.\\
            \hline
            \textsl{Massivemorph} & \\
            \hline
            \textsl{Flint} & \\
            \hline
            \textsl{Gold Moon Coin} & Activate it to give \textsl{Massivemorph} and \textsl{Flint} to your opponent. Again, we suppose they activate them. In particular, we suppose that the equip card is equipped to the link monster, and \textsl{Massivemorph} is activated targeting \textsl{Venemy Starving Dragon}.\\
            \hline
            \textsl{Raigeki} & \\
            \hline
        \end{tabular}
        }
    \end{table}

In this way remain in the Main Deck only 9 cards, which, together with \textsl{Double Summon}, which is already in Player~1's hand, and \textsl{Magician of Faith} on the field, correspond to the last 11 cards presented in the strategy in section \ref{secTuring}, and indeed, the last part of the construction concerning them is basically the same. As already mentioned, the opponent player are not really forced to play the cards we give them; however, as we are interested only in showing that the desired configuration is legal, this is not restrictive.

\subsection{How to make your opponent choose a number}

We now present how Player~2 can ``choose a number" by increasing their Life Points. Notice that althought Player~2 can increase their Life Points by an arbitrary number (multiple of 1000) they have two restrictions: first, by our assumption on the game rules, they can increase it only by a finite number each turn (since they can perform only a finite number of actions), and second Player~1 can always take away this possibily from them by using the card \textsl{Raigeki}.
Therefore, there is no risk that Player~2 continues to increase their Life Points, stalling the game.
To increase Life Points, Player~2 needs to change the target of the equip card \textsl{Flint} using the effect of \textsl{Flint Lock}, which not only allows them to change the target from \textsl{Clara \& Rushka, the Ventriloduo} to itself, but then allows them to swap the target using the other copy of itself an arbitrary number of times. Each time, changing the target will trigger the effect of \textsl{Morale Boost}, which increases by 1000 points Player~2's Life Points.\footnote{This is a famous combo named \emph{Flint Lock Loop}.} To avoid to destroy \textsl{Flint} during the activation of \textsl{Book of Eclispe}, we force Player~2 to choose as the last target of it \textsl{Clara \& Rushka, the Ventriloduo} that cannot be set face-down (recall that when an Equip Card is assigned to a monster, then it is destroyed if the monster changes its position to covered face-down). By ``forcing Player~2", in this case, we mean that the winning strategy that we build make Player~1 attack until reaching the winning condition if the opponent do not do such action.

We observe that this group of actions almost does not interfere with the two loops used by Player~1 to increase and decrease the Spell Counter on the card \textsl{Magical Citadel of Endymion} (because changing the target of a spell card does not count as a new activation). The only problem is that, having four monsters on his side of the field, Player~2, during the End Phase of Player~1’s turn, have to draw the same amount of cards by the effect of \textsl{Book of Eclispe}. To avoid this side effect, Player~1 has to activate at the beginning of each turn the effect of \textsl{Starving Venemy Dragon} on \textsl{Protector of the Sanctuary}; in this way, they are free to draw cards from the main deck (performing the loop to increase Spell Counters) and avoid the side effect of \textsl{Book of Eclipse}. 
\section{What is the real complexity of our problem?}\label{secCoCompl}

A straightforward computations shows that the set $\{(\overline{C}, S) \mid \overline{C} \overset{S}{\rightarrow}V \}$ is $\Pi^1_1$, indeed:
\begin{align*}
  \{(\overline{C}, S) \mid \overline{C} \overset{S}{\rightarrow}V \} =& \{ (\overline{C}, S) \mid \
  \overline{C} \in \mathcal{R} \land  \forall p \in \baire  ([\forall n < \operatorname{l}(\overline{C})(p(n)=\overline{C}(n)) \land \\ 
  &\forall n \ge \operatorname{l}(\overline{C}) [\exists M \in \omega ( M \text{ codes a move } \land t(p(n), p(n+1), M) )\land  \\
  &\land (p(n) \neq p(n+1) \lor p(n+1) \neq p(n+2)) \land I(p \restriction n+1) \\
  &\implies S(p \restriction n+1) = p(n+1)] \\
  &\implies \exists m \in \omega (p(m) \text{ is a winning state for Player~1})]) \}  
\end{align*}
where $t \subseteq \omega^3$ is the computable predicate that checks whether a transition is legit (which exists by our assumptions), and $I \subseteq \omega^{<\omega}$ is the computable predicate checking if Player~1 can do something. Therefore, as checking if a current configuration of the game is winning is clearly computable, the thesis follows.

It is natural to ask whether this upperbound is sharp. We will devote the rest of this section to proving that this is the case, as we can reduce to it the following $\Pi^1_1$-complete set.

\begin{proposition}[\protect{\cite[Exercise I.32]{Montalban_2026}}]\label{exmontalban}
The set $\mathrm{NIS}$ of indices $e \in \omega$ for which there is no infinite sequence $\langle a_n: n \in \omega\rangle$ such that $e$-th Turing machine $\varphi_e$ on the input $a_{n+1}$ outputs $a_{n}$ for all $n \in \omega$ is $\Pi^1_1$-complete.
\end{proposition}

\begin{proof}
See appendix \ref{secApp}.
\end{proof}

\begin{theorem1.3}
Determining whether a computable strategy in the \emph{Yu-Gi-Oh! TCG} is winning is a $\Pi^1_1$-complete problem.
\end{theorem1.3}

\begin{proof}
Similarly to Theorem 1.1, we prove that $\{(\overline{C}, S) \mid \overline{C} \overset{S}{\rightarrow}V \}$ is $\Pi^1_1$-hard by presenting a many-reduction from $\mathrm{NIS}$ to it. Consider the function $f: \omega \to \mathcal{R} \times \mathcal{S}$ that associate to each $e$ the pair $(\mathbf{C}_\ast,\mathbf{S}_e)$, where $\mathbf{C}_\ast$ is the legal run which leads to the configuration explained in Subsection 4.1 and $\mathbf{S}_e$ is the strategy starting from the run $\mathbf{C}_\ast$ and then ask to opponent for a number which is read as a triple $\langle a_1, a_2, t_1 \rangle$ then compute $\varphi_e(a_2)[t_1]$ and:
\begin{itemize}
  \item if $\varphi_e(a_2)[t_1]\!\!\downarrow$, check whether $\varphi_e(a_2) = a_1$, if so then ask to opponent for a number which is read as a pair $\langle a_3, t_2 \rangle$ and repeat the previous control for the triple $\langle a_2, a_3, t_2 \rangle$, otherwise activate \textsl{Raigeki} (to stop the possibility of increase lifepoints) and attack opponent's Life Points directly until reaching the winning condition.
  \item otherwise activate \textsl{Raigeki} and attack the opponent's Life Points until they become 0.
\end{itemize}
It is immediate to prove that this function is a reduction, indeed if $e \in \mathrm{NIS}$ then for any $a \in \baire$ resulting from a possible play of the opponent for some $n \in \omega$ either $\varphi_e(a_{n+1})\!\!\downarrow \neq a_n$ or $\forall s \in \omega \varphi_e(a_{n+1})[s]\!\!\uparrow$, and hence it is clear that an possible run ends in a victory for Player 1. Similarly, if $e \notin \mathrm{NIS}$, then there is some $a \in \baire$ such that
\[\forall n \in \omega \exists t_n \in \omega \varphi_e(a_{n+1})[t_n]\!\!\downarrow = a_n \]
therefore, Player 2 cannot lose while playing the pair $\langle a_{n+1}, t_n \rangle$.

As for Theorem 1.1, such a function is computable, thus $\mathrm{NIS} \le_m \{(\overline{C}, S) \mid \overline{C} \overset{S}{\rightarrow}V \}$.
\end{proof}

\begin{remark}
  We point out that the requirement about the ``memory" of the strategy of all previous turns is essential for our result. Indeed, Player~1's strategy needs to know the increase in Life Point happened during the previous turn. Although this may seem strange for \emph{Yu-Gi-Oh! TCG}, as for most of the cards, seems sufficient to require only memory for the actions that happened during the same turn (in addition to the card played and on the field, and the active effects). However, this is not really the case as there are cards that require memory of the actions that happened during the previous turn as \textsl{Life Absorbing Machine}, so our assumption is reasonable.
\end{remark}

\section{A more general problem}\label{secDST}

\noindent Now we move our attention to non-computable strategies.
Removing the word ``computable" from the definitions in section 3, we can ask when a strategy is winning from a certain run. With abuse of notation, we use $\overline{C} \overset{S}{\rightarrow}V$
to indicate when a strategy $S$ is winning starting from the run $\overline{C}$.

\begin{notation}
 Let $\overline{C}$ be a run and $n$ be the coding of a legal configuration reachable from $\overline{C}$. With $\overline{C}^\frown n$, we denote the run obtained from adding the move coded by $n$ after the last move of $\overline{C}$.
\end{notation}

We view the set of all strategies as a subspace of a Polish space (more precisely as subset of a closed subset of the Baire space $\omega^\omega$). Indeed, any strategy can be seen as a tree where each finite branch is of the form $\overline{C}^\frown i^\frown S(\overline{C}^\frown i)$ for some legal run $\overline{C}$ not in a winning condition, $i$ code of a legal move of Player~2, and $S(\overline{C}^\frown i)$ response of Player~1's strategy, or of the form $\overline{C}$ without children if $\overline{C}$ is a legal run with a winning condition. Moreover, the set of all trees $\operatorname{Tr}$ can be coded as a $0$-dimensional Polish space (see \cite[Exercise 4.32]{kechris}).

\begin{remark}
Since we see strategies as trees, the problem asking $\overline{C} \overset{S}{\rightarrow}V $ is the same as asking if the subtree $\{s \mid \overline{C}^\frown s \in S \}$ (also said the \emph{localization of $S$ at $\overline{C}$}) has no infinite branches and all leaves present a victory condition for Player~1.
\end{remark}

\begin{remark}
    By basically the same computation done in section \ref{secCoCompl}, we get that $\{(\overline{C}, S) \mid\overline{C} \overset{S}{\rightarrow}V \}$ is in $\boldsymbol{\Pi}^1_1(\omega^{<\omega} \times \operatorname{Tr})$.
\end{remark}

We want to show this set is $\boldsymbol{\Pi}^1_1$-complete. To do so, we define a Wadge-reduction from $\mathrm{WO}$, the subset of $\mathrm{LO}$ (i.e., of all countable linear orders on $\omega$) that are well-founded, which is a well-known $\boldsymbol{\Pi}^1_1$-complete set (see \cite[Theorem 27.12]{kechris}).

\begin{theorem1.4}
Determining whether a (not necessarily computable) strategy in \emph{Yu-Gi-Oh! TCG} is winning is a $\boldsymbol{\Pi}^1_1$-complete problem.
\end{theorem1.4}

\begin{proof}
    We define a continuous function $f: \mathrm{LO} \to \omega^{<\omega} \times \mathrm{Tr}$. For every $x \in \mathrm{LO}$, $f(x)=(\mathbf{C}_\ast, \mathbf{S}_x)$ where $\mathbf{C}_\ast$ is the legal run that leads to the configuration explained in Subsection 4.1, and the strategy $\mathbf{S}_x$ is defined as follows. 
    
    Starting from the run $\mathbf{C}_\ast$ (outside of the localization, it is not relevant if it is not winning, so we do not define it), we ask the opponent for a number which is read as a pair $\langle m,n \rangle$, then compute $x(m,n)$ and:
\begin{itemize}
  \item if $x(m,n)=1$, ask the opponent for a number $l \neq m$ and repeat the previous control for the triple $\langle l,m \rangle$; otherwise, activate \textsl{Raigeki} (to stop the possibility of increasing lifepoints) and attack the opponent's Life Points directly until they become 0.
  \item otherwise activate \textsl{Raigeki} and attack the opponent's Life Points until they reach 0.
\end{itemize}

Clearly, this is a reduction since $x \in \mathrm{WO}$ if and only if $\mathbf{C}_\ast \overset{\mathbf{S}_x}{\rightarrow}V$. Moreover, notice that given an $x \in \mathrm{LO}$, the output of our construction $(\mathbf{C}_\ast, \mathbf{S}_x)$ is actually computable in a uniform way, so this reduction is continuous. \qedhere
\end{proof}
\appendix

\section{Proof of Proposition \ref{exmontalban} }\label{secApp}
\newcommand\func[3]{#1:#2\to #3}

When $e\in\omega$ and $x\in\omega^\omega$, we write $\varphi_e^x(n)\downarrow$ if the $e$-th computable function with oracle $x$ converges on the input $n$, and $\varphi_e^x(x)[t]\downarrow$ if it converges in at most $t$ steps. Notice that this implies that the output may only depend on the initial string $(x_0,\dots ,x_{t-1})$ and not on further values of the oracle.

We shall need the following classic result, which is a corollary of \cite[Theorem 1]{Kleene1956}:

\begin{fact}[Kleene's Normal Form for $\Pi^1_1$]Given $S \in \Pi^1_1(\omega)$, then it can be described by a formula of the form
\[n\in S\:\Longleftrightarrow\:\forall x\in\omega^\omega \exists m \in \omega(P(x, m, n))\]
with a computable predicate $P$, and hence one can find an index $z_0 \in \omega$ so that
\[n\in S\:\Longleftrightarrow\:\forall x\in\omega^\omega(\varphi^x_{z_0}(n)\downarrow).\]
\end{fact}

\begin{proposition5.1}[\protect{\cite[Exercise I.32]{Montalban_2026}}]The set
\[\mathrm{NIS}=\{e\in\omega \mid \neg\exists x \in\omega^\omega\forall n \in \omega (\varphi_e(x_{n+1})=x_n)\}\]
is $\Pi^1_1$-complete.
\end{proposition5.1}

\begin{proof}
Given $s = (s_0, \dots, s_{m-1}) \in \omega^{<\omega}$, we denote by $\langle s_0,\dots,x_{m-1}\rangle$ the natural number corresponding to this string via an effective bijection, and the concatenation, indicated as before using $s^\frown t$, is used in both contexts. $\mathrm{NIS}$ is $\Pi^1_1$ by definition.

Let $S\subseteq\omega$ be some $\Pi_1^1$ set, and $z_0$ be an index for it (given by the previous fact). We now show a many-one-reduction from $S$ to $\mathrm{NIS}$. 

When $s=\langle s_0\dots,s_t\rangle\in\omega$, 
we write $\varphi_{z_0}^s(n)\downarrow$ iff $\varphi^x_{z_0}(n)[t]\downarrow$ 
for some/any $x\in\omega^\omega$ such that $\forall i < t (x_i=s_i)$. Do notice that checking whether $\varphi_{z_0}^s(n)\downarrow$ is computable. Now we define $f:\omega\times\omega \to \omega$ as
\[
    f(n, \langle s_0,\dots,s_{m-1}\rangle)=\begin{cases}
\langle s_0,\dots,s_{m-2}\rangle&\text{if}\quad n>0\:\land\:\neg\varphi_{z_0}^{\langle s_0,\dots,s_{m-1}\rangle}(n)\downarrow\\
\uparrow &\text{otherwise}
\end{cases}
\]
As $f$ is computable, using the S-m-n theorem, we can find a computable map $g: \omega \to \omega$ such that $f(n,\cdot)=\varphi_{g(n)}$. We show that $g(n)\in\mathrm{NIS}$ iff $n\in S$, so $S\le_m \mathrm{NIS}$.

Suppose first that $g(n)\not\in \mathrm{NIS}$. So there exists some sequence $q \in \omega^\omega$ such that $\forall m \in \omega (\varphi_{g(n)}(q_{m+1})=f(n,q_{m+1})=q_m)$. 
In particular, by the definition of $f$, it must be the case that $\neg\varphi_{z_0}^{q_m}(n)$ for every $m\in \omega$. 
Moreover, it is clear that $q_{m+1}=q_m^\frown q'_{m}$ for some $q'_{m}\in\omega$. 
As any element of $q \in \baire$ codes a finite string, consider the element of the Baire space which extends all of them: $\tilde{q} \coloneq \bigcup_{n \in \omega} q_n$. 
Clearly, it cannot be the case that $\varphi^{\tilde{q}}_{z_0}(n)\downarrow$, for otherwise it would be that $\varphi^{\tilde{q}}_{z_0}(n)[t]\downarrow$ for some $t>0$ and hence $\varphi^{q_t}_{z_0}(n)\downarrow$, which is a contradiction. So $n\not\in S$.

Conversely, assume $n\not\in S$, and pick $x \in \omega^\omega$ so that $\varphi^x_{z_0}(n)\uparrow$. For $m\in\omega$, let $q_m=\langle x_0,\dots,x_{m-1}\rangle$ (and $q_0=\langle\rangle$). It is immediate that $\forall m \in \omega(f(n,q_{m+1})=q_m)$: indeed, since $q_{m+1}=q_m^\frown x_m$, it suffices to check that $\neg\varphi_{z_0}^{q_{m+1}}(n)\downarrow$. But if it were not so, then in particular $\varphi^x_{z_0}(n)[m]\downarrow$, which is a contradiction.
\end{proof}

\printbibliography

@article{mtgTuringComp,
  author = {Churchill, Alex and Biderman, Stella and Herrick, Austin},
  title = {Magic: The Gathering Is Turing Complete},
  booktitle = {10th International Conference on Fun with Algorithms (FUN 2021)},
  pages = {9:1--9:19},
  year = {2020},
  volume = {157},
  publisher = {Schloss Dagstuhl -- Leibniz-Zentrum f{\"u}r Informatik},
  doi = {10.4230/LIPIcs.FUN.2021.9}
}

@article{Ward2009MonteCS,
  title={Monte Carlo search applied to card selection in Magic: The Gathering},
  author={Peter I. Cowling and Colin D. Ward},
  journal={2009 IEEE Symposium on Computational Intelligence and Games},
  year={2009},
  pages={9-16},
  doi={10.1109/CIG.2009.5286501}
}

@article{Ward2012MonteCS,
  author={Cowling, Peter I. and Ward, Colin D. and Powley, Edward J.},
  journal={IEEE Transactions on Computational Intelligence and AI in Games}, 
  title={Ensemble Determinization in Monte Carlo Tree Search for the Imperfect Information Card Game Magic: The Gathering}, 
  year={2012},
  volume={4},
  number={4},
  pages={241-257},
  doi={10.1109/TCIAIG.2012.2204883}
}

@misc{HardArithmetic,
      title={Magic: the Gathering is as Hard as Arithmetic}, 
      author={Stella Biderman},
      year={2020},
      archivePrefix={arXiv},
      doi={10.48550/arXiv.2003.05119}
}

@book{kechris,
  title={Classical Descriptive Set Theory},
  author={Kechris, A.},
  doi={10.1007/978-1-4612-4190-4},
  series={Graduate Texts in Mathematics},
  year={1995},
  publisher={Springer New York}
}

@book{Montalban_2026, 
  place={Cambridge}, 
  series={Perspectives in Logic}, 
  title={Computable Structure Theory: Beyond the Arithmetic}, 
  publisher={Cambridge University Press}, 
  author={Montalbán, Antonio}, 
  year={2026}, 
  collection={Perspectives in Logic},
  doi={1017/9781108780568}
}

@misc{rulebook,
    title={Yu‑Gi‑Oh! TCG Rulebook}, 
    year={2021},
    author={Konami},
    url={https://www.yugioh-card.com/eu/play/tcg-rulebook/}
}

@article{Kleene1956,
	author = {S. C. Kleene},
	doi = {10.2307/2268429},
	journal = {Journal of Symbolic Logic},
	number = {4},
	pages = {411--412},
	publisher = {Association for Symbolic Logic},
	title = {Hierarchies of Number-Theoretic Predicates},
	volume = {21},
	year = {1956}
}
\end{document}